\documentclass[dvips,11pt ,twoside]{amsart}  
\usepackage[left=3.1cm,right=3.1cm,top=3.1cm,bottom=3.1cm]{geometry}
\usepackage{graphicx}
\usepackage{multirow}
\usepackage{hyperref}
\usepackage{bbm}
\usepackage{amsmath,amssymb,amsfonts, amsthm,epsfig}
\usepackage{verbatim}
\usepackage{enumerate}

\numberwithin{equation}{section}

\newtheorem{thma}{Theorem}[section]
\newtheorem{lemma}[thma]{Lemma}
\newtheorem{coro}{Corollary}

\renewcommand{\l}{{\lambda}}
\newcommand{\g}{{\gamma}}

\newcommand{\E}{{\mathbb{E}}}
\renewcommand{\P}{{\mathbb{P}}}
\newcommand{\R}{{\mathbb{R}}}

\newcommand{\Z}{{\mathbb{Z}}}

\newcommand{\half}{{\frac{1}{ 2}}}
\newcommand{\e}{{\epsilon}}

\renewcommand{\half}{{\frac{1}{2}}}

\renewcommand{\(}{\begin{equation}}
\renewcommand{\)}{\end{equation}}

\begin{document}

\begin{abstract}
The dimer model on a planar bipartite graph can be viewed as a random surface measure. We study these fluctuations for a dimer model on the square grid with two different classes of weights and provide a condition for their equivalence. In the thermodynamic limit and scaling window,  these height fluctuations are shown to be non-Gaussian.

\end{abstract}

\title{The Height fluctuations of an off-critical dimer model on the square grid}

\author{Sunil Chhita}
\address{Royal Institute of Technology KTH, Stockholm, Sweden}

\maketitle

\section{Introduction} 

In recent years,  two dimensional grid based  statistical mechanical models have been the focus of much research.  These discrete models often possess a {\it critical point}, where there is an abrupt change of behavior, called the {\it phase transition}.    At this point, the models can exhibit macroscopic randomness.  This leads to studying the {\it scaling limit}, i.e. the behavior of the model when the lattice spacing tends to zero.

Before taking the scaling limit,  consider the {\it off-critical model} which is a perturbation from the critical point.  These types of models, in general, do not feature any macroscopic randomness and taking the scaling limit of such a model gives a trivial measure. However, one could consider the behavior when the lattice size tends to zero while simultaneously moving the parameters towards the critical point.  We call this limit a {\it scaling window}.    By choosing an appropriate re-scaling, we often find non-trivial scaling window measures for off-critical statistical mechanical models.  Studying these measures sometimes provides additional information about the critical point behavior.

In this paper, we focus on the dimer model (for the original paper, see \cite{Kas:61}) on a bipartite planar graph, which is a particular two dimensional statistical mechanical model, defined as follows for a finite graph. Call a {\it dimer configuration} to be a subset of edges in which each vertex is incident to one edge.  The {\it dimer model} is a probability measure on the set of all dimer configurations of the underlying graph.  Weights can be applied to the edges so that the probability of a configuration is proportional to the product of the edge weights.

\subsection{Overview of Setup} \label{sec:intro:briefsetup}


Let $G$ be an infinite planar bipartite graph.   As $G$ is bipartite,  we can define a {\it height function} which is in one-to-one correspondence (up to height level) with each dimer configuration (see \cite{Ken:97}).  This gives the viewpoint that the dimer model is the graph of a random function from $\R^2 \to \R$, i.e. a random surface (see \cite{Fun:05, she:05} for more details of random surfaces). A precise definition of the {\it height function} for the dimer model can be found in \cite{Ken:09}.  Informally, the height function can be defined as follows.


Suppose $d$ is the number of edges incident to each vertex and that $d$ is fixed for $G$. The height function is a continuous function from $\R^2$ to $\R$ which is integer valued at the center of each face of $G$.  The height change between the centers of two faces sharing an edge is $\pm (d-1)$ if there is a dimer covering the shared edge or $\mp 1$ if there is no dimer covering that shared edge.  We choose the following sign convention:  suppose that the black vertex is to the left when traversing the shared edge between two faces.  The height change along the traversal is $(d-1)$ if there is a dimer covering the shared edge and -1 otherwise.  With this setup, it can easily be seen that the total height change at the centers of the faces surrounding each vertex is zero.

Define the toroidal graph $G_n$ to be the quotient of $G \slash n \Z^2$, that is, the action of $n \Z^2$ on $G$ by horizontal and vertical translations.  Let $\mathcal{M}(G_n)$ denote the set of all dimer coverings on $G_n$.  Define a probability measure $\mu_n$ for $M \in \mathcal{M}(G_n)$ so that the probability of observing $M$ is proportional to the product of all the edge weights.   For dimer models on periodic graphs, there is a limiting measure $\mu$ when the system size is sent to infinity.  This limiting measure is a {\it Gibbs} measure and is called the {\it thermodynamic limit}.

As mentioned above, the scaling window represents the limiting measure (of an observable) when the lattice spacing tends to zero while simultaneously moving the parameters to the critical point.  We use the following notation for scaling windows: for an order $\e$ perturbation from criticality, the {\it scaling window} $(\e^\alpha,\e^\beta)$ of a measure is the limiting measure when re-scaling the horizontal axis by $\e^\alpha$, the vertical axis by $\e^\beta$ and $ \e \to 0$ for $\alpha,\beta>0$.   In other words, a box consisting of $1/\e^{(\alpha +\beta)}$ lattice points is sent to a box with continuum area 1 under the scaling window  $(\e^\alpha,\e^\beta)$.  In the literature, scaling windows have appeared in different guises such as {\it critical windows} or {\it intermediate phases}.   In Section \ref{sec:intro:related}, we briefly survey some of these results for other two dimensional discrete models and provide some motivating remarks.

\subsection{Results}

 This paper considers two differently weighted dimer models of the square grid. We say that the dimer model is {\it flipped}  if the edge weights alternate between $r_1$ and $r_3$ along a vertical line and between $r_2$ and $r_4$ along a horizontal line.  Let $\mu_f(r_1,r_2,r_3,r_4)$ denote the flipped dimer model with edge weights $r_1,r_2,r_3$ and $r_4$.  We say that the dimer model is {\it drifted} if the edge weights are $s_1, \dots s_4$ for edges surrounding vertices $\mathbb{Z}_e^2$, the even sub-lattice of $\mathbb{Z}^2$.  Let $\mu_d(s_1,s_2,s_3,s_4)$ denote the drifted dimer model with edge weights $s_1,s_2,s_3,s_4$.

\begin{thma} \label{intro:thm:equivmeas}
	 The measures $\mu_f(r_1,r_2,r_3,r_4)$ and $\mu_d(s_1,s_2,s_3,s_4)$ are equivalent in the thermodynamic limit when $s_1=r_1/r_4$, $s_2=r_2^2/(r_1 r_4)$, $s_3=r_3^2/(r_1 r_4)$ and $s_4= r_4/r_1.$
\end{thma}

Using the KPW-Temperley correspondence \cite{kpw:00}, the drifted dimer model is in bijection with random directed spanning trees with the root at infinity.  The random spanning tree can be constructed using Wilson's algorithm \cite{Wil:96} with random walk probabilities $(s_1,s_2,s_3,s_4)/(s_1+s_2+s_3+s_4)$ for the north, east, south and west steps.   The flipped dimer model also defines a spanning tree (and dual tree) with the root at infinity but this tree cannot be constructed using Wilson's algorithm with a massive random walk.  It follows from Theorem \ref{intro:thm:equivmeas} that these two trees are equal in distribution.  It is not known whether there is combinatorial construction which gives a bijection between these two trees.

The dimer model $\mu_f(1,1,1,1)$ (or $\mu_d(1,1,1,1)$) is a critical statistical model \cite{Ken:01}.    We consider a perturbation of this dimer model with flipped weights and study the behavior of the height fluctuations in the scaling window.   As a consequence, this provides results regarding the height fluctuations of the drifted dimer model (in the scaling window).  We find that

\begin{thma} \label{intro:thm:non-gauss}

In the thermodynamic limit,  the height fluctuations for the flipped dimer model $\mu_f(1,1-\l_1 \e , 1-\l_2 \e,1)$ in the scaling window $(\e,\e)$ are non-Gaussian.

\end{thma}

The proof of non-Gaussian fluctuations is based on showing that the fourth moments of the height function cannot be written as the sum of a product of paired moments, i.e. does not satisfy Wick's Theorem.   Note that in the limit as $\l_1$ and $\l_2 \to 0$, the fluctuations are Gaussian and are given by the Gaussian Free Field \cite{Ken:00}.  If $\l_1$ and $\l_2 \to \infty$, the fluctuations are given by White Noise.  The probability distribution of the fluctuations of the height function of $\mu_f(1,1-\l_1,1-\l_2 \e,1)$ under the scaling window $(\e,\e)$ remains an open problem.

\subsection{Related Results and Motivating Remarks} \label{sec:intro:related}

We provide some of the (recent) results that are closely connected with scaling limits and windows of the dimer model and random directed spanning trees.

Richard Kenyon, in \cite{Ken:00} and \cite{Ken:01}, showed that the scaling limit of the height fluctuations of the  critical dimer model on the square grid is  conformally invariant and is given by a  Gaussian Free Field  (see  \cite{she:07} for more details on the GFF).   This is the same measure if we take the scaling window $(\e,\e)$ followed by the limit $\l_1, \l_2 \to 0$ for $\mu_f(1,1-\l_1 \e,1-\l_2 \e,1)$.  Recently, there have been many different results such as \cite{SS:09} proving that the contours of the discrete Gaussian free field converge to SLE(4) (Schramm Loewner Evolution) in the scaling limit.   SLE(4) is also the conjectured distribution of paths formed from the scaling limit of the critical double dimer model.  Excellent surveys of SLE can be be found in \cite{law:05} and \cite{Wen:04}.

The uniform spanning tree has been of considerable interest and many results can be found in \cite{Blps:01}.  The KPW-Temperely bijection \cite{kpw:00} shows that the uniform spanning tree is in bijection with the critical dimer model.  This bijection allows some further computations, which are accessible using the dimer model technology.  Indeed, \cite{Kenb:00} provided the first rigorous computational evidence  that the uniform spanning tree has some conformal structure.  It has been shown that the scaling limit of the Peano curve of the 2D uniform spanning tree is SLE(8) (\cite{sch:00}).

Scaling windows of other discrete models have also been considered.  In the literature, these have often been called near-critical or intermediate phase.  For 2D site percolation on a triangular lattice, \cite{nw:09} classified the different scaling windows while \cite{Cam:09} shows that there are an infinite number of small loops for near-critical percolation under a certain choice of scaling window.

There exists scaling windows of models whose scaling limits give SLE.  These objects are called Massive SLE  \cite{MS:09}.  They can be constructed by perturbing the fugacity and performing the scaling window with respect to the fugacity.  For example, the off-critical LERW features loops with a small weight and the lattice is then re-scaled while simultaneously sending the small weight to zero (see \cite{BBK:08}).  It is not known whether the scaling windows of the dimer models studied in this paper are connected with Massive SLE.

Non-Gaussian fluctuations are not uncommon for 2D statistical mechanical models perturbed from criticality.  
For example, the so-called 2D $\phi^4$-model has non-Gaussian fluctuations in the scaling limit \cite{BFS:83}.

The motivation behind Theorem \ref{intro:thm:non-gauss} was to find properties of the height fluctuations between the scaling limit of the critical and off-critical dimer model on the square grid.  In order to determine whether the off-critical model under re-scaling provided some non-trivial behavior, we needed to send the parameters of the model to the critical point while simultaneously sending the mesh size to zero which is defined as the scaling window in Section \ref{sec:intro:briefsetup}.  The choice of the scaling window parameter exponents $\alpha$ and $\beta$ are chosen so that the correlation length of height function is finite under re-scaling.  Other symmetric choices of the scaling window parameter exponents leads to the  fluctuations governed by either the Gaussian Free field or White Noise depending on whether the re-scaling is too slow or too fast.

Finally, many physical models are studied away from the critical point using  perturbation theory with the goal that they belong to some universality class.  For an example related to the dimer model on the square grid,    \cite{Las:97} considers the six-vertex model close to the critical point and shows the scaling under a certain limit is the same as the critical sine-Gordon model \cite{Luk:97}.  We believe that the model studied in this paper belongs to the universality class of free field theory deformed by a certain operator but the exact details of this remains open.

\subsection{Overview of the Paper}
The paper is organized as follows. In Section \ref{sec:setup}, we provide some background information about the dimer model, such as, the Kasteleyn method.  The proof of Theorem \ref{intro:thm:equivmeas}, the bijections of the discrete models as well as the operators for the underlying random walks are given in Section \ref{sec:discrete}.  In Section \ref{sec:scalingwindow}, we compute the relevant Green's functions for the operators, which allow us to prove Theorem \ref{intro:thm:non-gauss}. In the same section, we also provide the amoeba interpretation of the scaling window.

\section{Setup} \label{sec:setup}

\subsection{Partition Function, Local Probabilities and Notation}  \label{dimer:sect:partition}

We refer to the graph $G_1$ as the fundamental domain. Orient the edges of $G_1$  so that the number of counter clockwise edges per face is odd.     This is called the {\it Kasteleyn orientation}.    Assume that $G$ (the infinite planar graph) has the periodic Kasteleyn orientation induced by $G_1$.  As $G$ is $\mathbb{Z}^2$ periodic, we can write any vertex of $G$ as a vertex of $G_1$ plus  a translation of the fundamental domain, where the first co-ordinate of the translation is in the direction of the vector $(1,0)$ and the second co-ordinate is in the direction of the vector $(0,1)$ (i.e.  $v+(x,y)$ is the vertex $v$ in the fundamental domain translated by  $(x,y)$).  For example, the fundamental domain of the square octagon lattice with the diagonal edges assigned weight  $t$ and off diagonal edges assigned weight $1$ is depicted in Figure \ref{setup:fig:fund}.

\begin{figure}[h!]
\begin{center}
\includegraphics[height=3in]{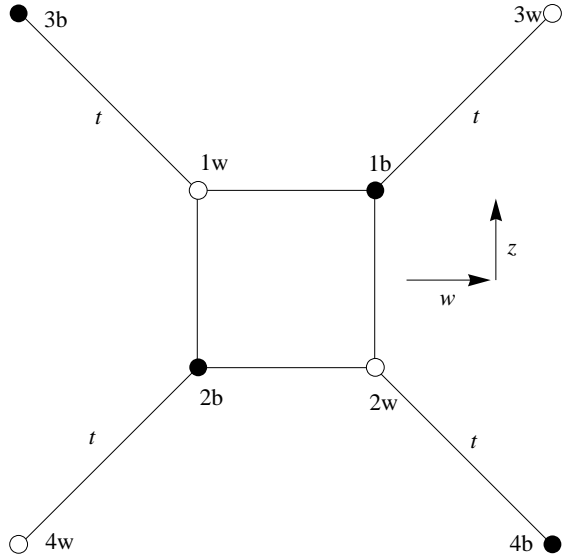}
\caption{Fundamental Domain of the square octagon lattice.  Take Kasteleyn orientation with the convention that all the edges are directed from white vertices to black vertices except for the edges $(w_1 , b_2)$ and $(w_3, b_4)$ which are both oriented in the reverse direction. }
\label{setup:fig:fund}
\end{center}
\end{figure}

Let $\mathtt{B}$ and $\mathtt{W}$ represent the white and black vertices of a bipartite graph. For any bipartite graph, define a weighted adjacency matrix by $K: \mathtt{W} \times \mathtt{B} \to \R$, with $K(w,b)$ given by the weight and direction of the edge $wb$.  This matrix is called the {\it Kasteleyn Matrix}.  For translational invariant graphs, it is easier to work with a Fourier transform as this gives information about the {\it free energy} and provides a tractable way of computing edge probabilities.   We now setup the Kasteleyn matrix on the torus using the method outlined in \cite{KOS:06}. 

Let $V(G_1)=\{w_i,b_j\}$ be the set of vertices of $G_1$ labelled as per Figure \ref{setup:fig:fund}. Let $\g_{1}$ (and resp. $\g_{2}$) be a path in the  dual of $G_1$, winding around the torus in the direction of the vector $(1,0)$  ((0,1)).   Write $e_{ij}$ for the weight of the edge $\overrightarrow{w_i b_j}$. For parameters $z$ and $w$, let $b_{ij}=z^{n_1}w^{n_2} e_{i j}$, where 
\begin{align}
	n_1 (n_2)=		\left( \begin{array}{ll} 
		1 & \overrightarrow{w_i b_j} \mbox{ crosses } \g_1 (\g_2) \mbox{ from below} \\
		0 & \overrightarrow{w_i b_j} \mbox{ does not cross } \g_1 (\g_2)  \\
		-1 & \overrightarrow{w_i b_j} \mbox{ crosses } \g_1 (\g_2) \mbox{ from above} \\
\end{array} \right.
\end{align}

Define $\tilde{K}_1$ to be the weighted adjacency matrix of $G_1$ with entry $(i,j)$ given by the weight connecting the $i^{th}$ white vertex to the $j^{th}$ black vertex and sign determined by the Kasteleyn orientation.  This means,
\begin{align}
	(\tilde{K}_1(z,w))_{i,j} = \left\{ \begin{array}{ll}
				b_{ij} & \mbox{if } w_i \sim b_j, \mbox{ and } w_i \to b_j \\
				-b_{ij} & \mbox{if } w_i \sim b_j, \mbox{ and } b_i \to w_j \\
				0 & \mbox{otherwise} \end{array}\right.
\end{align}
where $w_i \sim b_j$ means that $w_i$ and $b_j$ share an edge and $v_i \to v_j$ denotes an arrow from $v_i$ to $v_j$ in the Kasteleyn orientation.   For the above square octagon example, we have 
\begin{align}
	\tilde{K}_1(z,w)=
		\left( \begin{array}{cccc}
			1 &  -1 & t & 0 \\ 
			1 & 1 & 0 & t \\ 
			t & 0 & w & z \\ 
			0 & t & -\frac{1}{z} & \frac{1}{w}. 
\end{array}
\right)
\end{align}

The relevance of the above discussion is that we can write the partition function and probabilities of observing a local configuration  in the thermodynamic limit as expressions involving $\tilde{K}_1(z,w)$ (see \cite{KOS:06}).  

The characteristic polynomial for the fundamental domain, $P(z,w)$, is defined to be the determinant of $\tilde{K}_1(z,w)$ (see \cite{KOS:06}). For the square octagon example, this is given by
\(
	P(z,w)=4 + t^4 - \frac{t^2}{w} -t^2 w - \frac{t^2}{z} - t^2 z \label{sec:setup:charpoly}
\)
We can write the {\it free energy} as the following:
\begin{align}
 	\log Z =\frac{1}{(4 \pi ^2)} \int_{|z|=1} \int_{|w|=1} \log P(z,w) \frac{ dw}{w} \frac{dz}{z} \label{sec:setup:logZ}
 \end{align}

We can now give the {\it local statistics formula} \cite{Ken:97} which is given by the inverse Kasteleyn entries.   Let $E=\{e_1,\dots,e_l\}$ be a set of edges with $e_j=(w_j^*,b_j^*)$. 
\(
	\P(e_1, \dots, e_m) = \left( \prod_{j=1}^m \tilde{K}(w_j^*,b_j^*) \right) \det \left( (\tilde{K}^{-1} (b_i^*,w_j^*) \right))_{1 \leq i ,j \leq  m} \label{setup:localstats}
\)
where, assuming $b$ and $w$ are in the same fundamental domain,
\(
	K^{-1} (b,w+(x,y)) = \frac{1}{(2 \pi i)} \int_{\mathbb{T}^2} \tilde{K}^{-1} (z,w)_{b,w} w^{-x} z^{-y} \frac{dw}{w} \frac{dz}{z}.
\)

\subsection{Gauge Transformation and P-equivalence} \label{subsec:gauge}

We say that two weight functions, $v_1$ and $v_2$, are {\it gauge equivalent} if there are functions on the white vertices and black vertices, $f_1$ and $f_2$, such that for each edge $e=wb$,  $v_1(e)= f_1(b) f_2(w) v_2(e)$.  The mapping from $v_1$ and $v_2$ is called a {\it gauge transformation} and is measure preserving. 

We say that two dimer models with periodic weights are {\it P-equivalent} if the edge weights of the local configurations of each dimer model differ by a multiplicative constant dependent on the locations of each local configuration.   It is not necessarily true that P-equivalence preserves the measure.   

An example of P-equivalence is the following: consider a dimer model on the square grid which has edge weights $s^x$ for edges incident to the vertices $(2x,2y)$, edge weight $s^{x-1}$ for the edge connecting the vertex $(2x-2, 2y+1)$ to the vertex $(2 x-1,2y+1)$ and edge weights $s^x$ for the rest of the edges incident to the vertices $(2x-1,2y+1)$ for $x,y \in \Z$ and $s\not = 1$.   This dimer model is P-equivalent to the dimer model on the square grid with each edge having weight $s$ as shown in Figure \ref{ch1:intro:pic:localgauge}. Note that this is not a gauge transformation.

\begin{figure}[h!]
\begin{center}
\epsfsize220pt
$$\scalebox{1}{\includegraphics{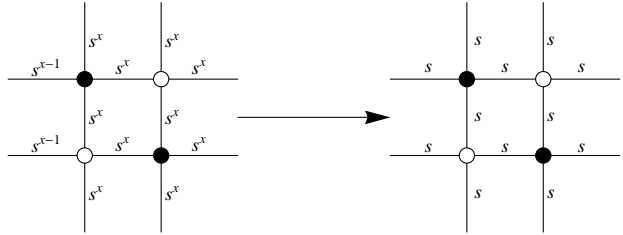}}$$
\caption[Local Gauge Transformation]{An example of P-equivalence.}
\label{ch1:intro:pic:localgauge}
\end{center}
\end{figure}

\subsection{Amoebas and Phase Diagrams}

The {\it amoeba} of $P(z,w)$, denoted $\mathbb{A}(P)$, is defined to be the set $\{(\log|z|,\log|w|) : P(z,w)=0 \}$ which forms curves in $\R^2$.  These curves split $\R^2$ into different regions, with each region describing the phase behavior of the measure of the height function for the dimer model, $\mu$, with characteristic polynomial $P(z,w)$, on some bounded domain.  For $\mu$, the height differences can be deterministic ({\it frozen}), stochastic with unbounded variance ({\it liquid}) or stochastic with bounded variance ({\it gaseous}).    It has been shown \cite{KOS:06}, that $\mu$ is frozen, liquid or gaseous, in the closure of the unbounded component of $\R^2 \backslash \mathbb{A}(P)$, in the interior of $\mathbb{A}(P)$ or in the closure of the bounded component of $\mathbb{A}(P)$.  The origin provides the specific location of the measure.  A more in depth discussion can be found in \cite{KOS:06,Ken:09}.   

An example of an amoeba the dimer model on the square octagon lattice, (\ref{sec:setup:charpoly}), can be found in Figure \ref{sec:fig:sqoctamoeba}.  For $t<\sqrt{2}$, the phase diagram features solid, liquid and gaseous phases  and the measure is in the gaseous phase.  At $t=\sqrt{2}$, there is no gaseous phase and the measure is in the liquid phase. The height fluctuations in the scaling limit are given by a Gaussian Free Field (\cite{kpw:00, Ken:01, KOS:06}).   For $t=\sqrt{2-\e}$, the height fluctuations in the scaling window $(\e,\e)$ represent the transition between the liquid and gaseous phases.

\begin{figure}[h!]
\begin{center}
\epsfsize220pt
$$\scalebox{0.75}{\includegraphics{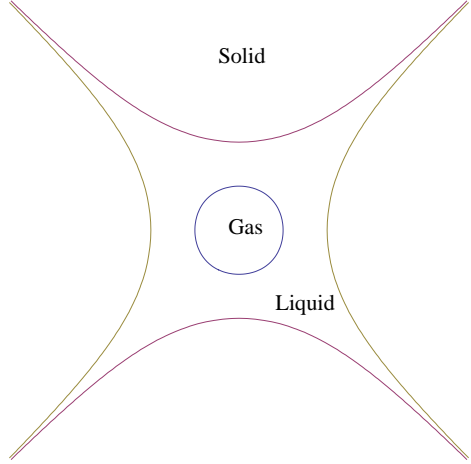}}$$
\caption{The amoeba of the dimer model on the square octagon lattice with characteristic polynomial (\ref{sec:setup:charpoly}), with phases labelled.}
\label{sec:fig:sqoctamoeba}
\end{center}
\end{figure}

\subsection{Some Notation}

For the dimer models on the square grid, we have the following notation: Let $\mathtt{W}$ and $\mathtt{B}$ denote the white and black vertices.  Let $\mathtt{B}_0$ denote the black vertices $\{(2x,2y): x,y \in \Z\}$, $\mathtt{B}_1$ denote the black vertices $\{(2x+1,2y+1): x,y \in \Z\}$,$\mathtt{W}_0$ denote the white vertices $\{(2x+1,2y): x,y \in \Z\}$ and $\mathtt{W}_1$ denote the white vertices $\{(2x,2y+1): x,y \in \Z\}$.



\section{The Equivalence of $\mu_f$ and $\mu_d$} \label{sec:discrete}

In this section, we prove Theorem \ref{intro:thm:equivmeas}.   By setting $s_1=r_1/r_4$, $s_2=r_2^2/(r_1 r_4)$, $s_3=r_3^2/(r_1 r_4)$ and $s_4=r_4/r_1$, the measures $\mu_f(r_1,r_2,r_3,r_4)$ and $\mu_d(s_1,s_2,s_3,s_4)$ are P-equivalent, as shown in Figure \ref{setup:fig:localgauge}.

\begin{figure}[h!]
\begin{center}
\epsfsize220pt
$$\scalebox{1}{\includegraphics{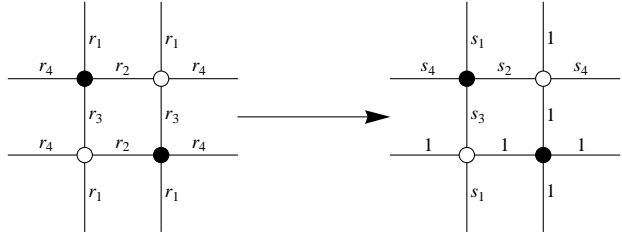}}$$
\caption[P-equivalence between flipped and drifted dimer model]{P-equivalence between the flipped dimer model to the drifted dimer model with $s_1=r_1/r_4$,  $s_2=r^2_2/(r_4 r_1)$, $s_3=r_3^2/(r_1 r_4)$ and $s_4=r_4/r_1$.}
\label{setup:fig:localgauge}
\end{center}
\end{figure}

Using the Fourier Transform method, the characteristic polynomials for both models can be found and are given by 
\(
	P_f(z,w)=r_2 r_4 \left(w+w^{-1} \right) + r_1 r_3 \left(z+z^{-1} \right)-(r_1^2+r_2^2 +r_3^2+r_4^2)
\)
for the flipped dimer model and
\(
	P_d(z,w)=s_2 w + s_4 w^{-1}+s_3 z +s_1z^{-1}-(s_1+s_2+s_3+s_4)
\)
for the drifted dimer model where $s_1=r_1/r_4$,  $s_2=r^2_2/(r_4 r_1)$, $s_3=r_3^2/(r_1 r_4)$ and $s_4=r_4/r_1$.  It is clear that 
\(
	P_d(z,w)=\frac{1}{r_1 r_4}P_f\left(\frac{r_3}{r_1} z , \frac{r_4}{r_2} w \right).
\)
Examples of the amoebas of $P_d$ and $P_f$ are depicted in Figure \ref{therm:fig:amoebas}.

\begin{figure}[h!]
\begin{align*}
\begin{array}{cc}
\scalebox{0.6}{\includegraphics{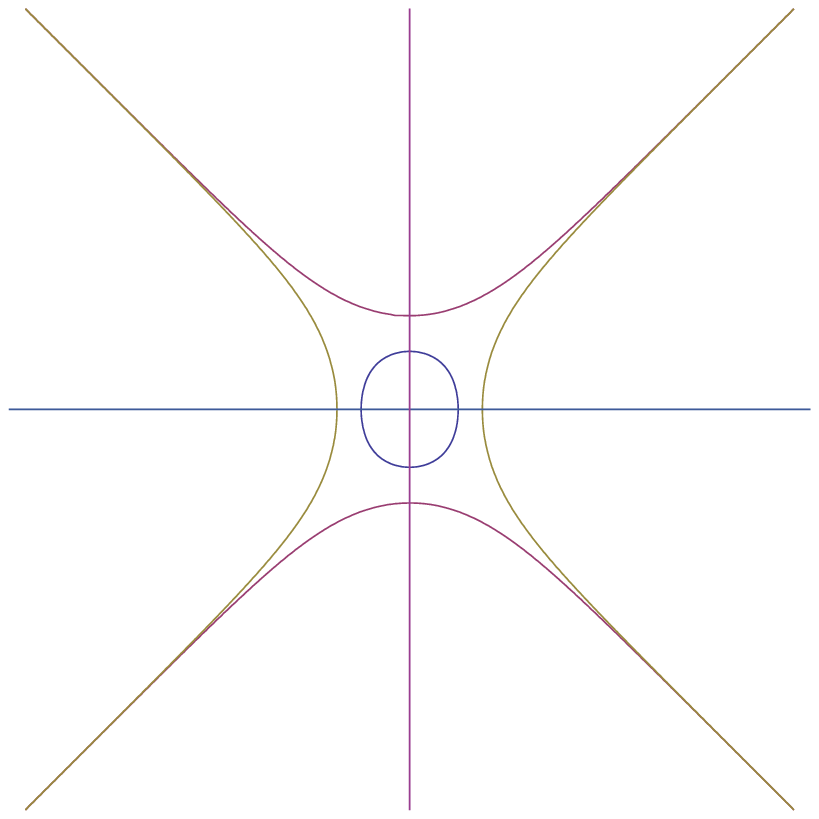}} &
\scalebox{0.6}{\includegraphics{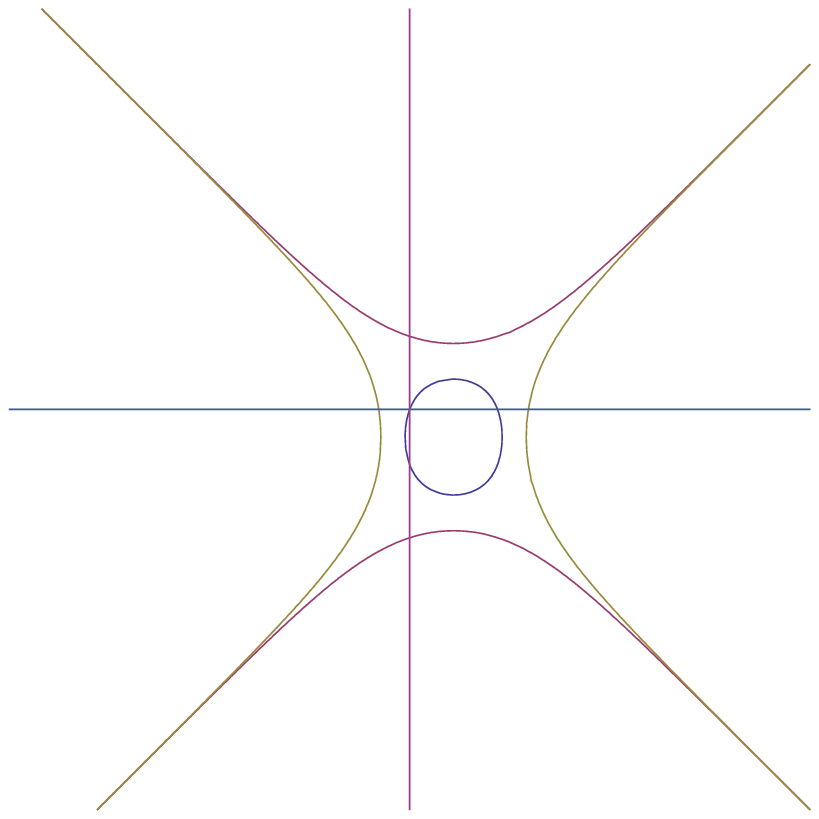}}
\end{array}
\end{align*}
\caption[Amoebas for the flipped and drifted dimer models]{The amoebas for the flipped dimer model (left) and the drifted dimer model (right).  The horizontal and vertical lines are the $x$ and $y$ axis. }
			\label{therm:fig:amoebas}
\end{figure}

\begin{proof}[Proof of Theorem \ref{intro:thm:equivmeas}]

Let $b_i \in \mathtt{B}_i$ and  $w_i \in \mathtt{W}_i$ for $i \in \{0,1\}$.  The fundamental domain for the square grid and the Kasteleyn orientation are given in Figure \ref{sqgrid:kastorient}.  Let $K_f(z,w)$ and $K_d(z,w)$ be the Kasteleyn matrices for the flipped and drifted dimer models, where the rows are indexed by $w_0$ and $w_1$ and the columns are indexed by $b_0$ and $b_1$.   These are given by
\begin{align}
	K_f(z,w)= \left( \begin{array}{cc}
	r_2 -r_4 z & r_3 -r_1 w \\
	r_3 - r_1 w^{-1} & -r_2+r_4 z^{-1}
	\end{array} \right)
\end{align}
and
\begin{align}
	K_f(z,w)= \left( \begin{array}{cc}
	s_2 -s_4 z & 1 - w \\
	s_3 - s_1 w^{-1} & -1+ z^{-1}
	\end{array} \right)
\end{align}
where the signs are determined by the Kasteleyn orientation and the appropriate edge weights can be found in Figure \ref{setup:fig:localgauge}

\begin{figure}[h!]
		\begin{center}
		\scalebox{0.6}{\includegraphics{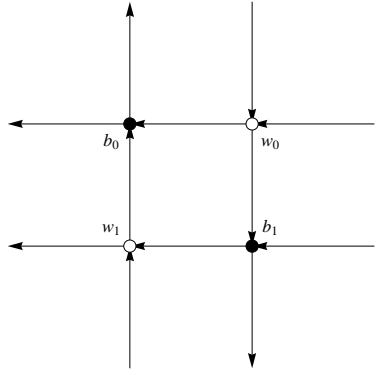}}
		\caption[Kasteleyn Orientation for the Square Grid.]{The Kasteleyn orientation and vertex labels for the square grid.}
		\label{sqgrid:kastorient}
		\end{center}
	\end{figure}

We have that
\(
	K_d(z,w)=D_1 . K_f\left(\frac{r_4}{r_2} z, \frac{r_3}{r_1} w \right) . D_2
\)
where
\begin{align}
	D_1= \left( \begin{array}{cc}
	\frac{1}{r_1 r_4} & 0 \\
	0&\frac{r_3}{r_2 r_1 r_4} 
	\end{array} \right)
\end{align}
and
\begin{align}
	D_2= \left( \begin{array}{cc}
	r_2 & 0 \\
	0 & \frac{r_1 r_4}{r_3}  \end{array} \right).
\end{align}

Let $\alpha=r_4/r_2$ and $\beta=r_3/r_1$ and let $K^{-1}_f$ and $K^{-1}_d$ represent the flipped and drifted inverse Kasteleyn matrices respectively.  Suppose that $w_i(x,y)$ represents  $w_i$ in the $(x,y)$ fundamental domain from $b_j$ for $i,j \in \{0,1\}$. In order to distinguish between the vertices, write $K^{-1}_f (b_i, w_j(x,y))$ for the inverse Kasteleyn entry between $b_i$ and $w_j(x,y)$.  We have
\(
	K^{-1}_d(b_0,w_0(x,y))=\frac{ \alpha^{x-1} \beta^y}{ r_1} K^{-1}_f(b_0,w_0(x,y)),
\)
\(
	K^{-1}_d (b_0,w_1(x,y))=\frac{\alpha^x \beta^{1+y}}{r_4}  K^{-1}_f(b_0,w_1(x,y)),
\)
\(
	K^{-1}_d (b_1,w_0(x,y))=\frac{\alpha^x \beta^{y-1}}{r_1}  K^{-1}_f(b_1,w_0(x,y)),
\)
and
\(
	K^{-1}_d (b_1,w_1(x,y))=\frac{\alpha^{x+1} \beta^{y}}{r_4}  K^{-1}_f(b_1,w_1(x,y)).
\)
It remains to check that any configuration of edges is equivalent using the local statistics formula (\ref{setup:localstats}).  The above values have the property that for $i,j \in \{0,1\}$,
\(
	w_k(e)K^{-1}_f (b_i,w_j(x,y))=w_d(e) \alpha^x \beta^y K^{-1}_d(b_i,w_j(x,y)) \label{thermo:eqn:multiplic}
\)
where $e$ is the edge incident to $b_i$ in the same fundamental domain, $w_k$ and $w_d$ are the edge weights of $e$ in the flipped and drifted dimer models.

For the plane, each entry in the inverse Kasteleyn matrix for the flipped dimer model is the same as the drifted dimer model up to a multiplicative factor dependent on the distance of the two vertices (\ref{thermo:eqn:multiplic}).   Each combination of edges is given by a determinant of the inverse Kasteleyn matrix.   However, the sum  of the exponents  (in the drifted dimer model) for each term in the expansion of the determinant is zero.  Hence, the determinant expansions for the flipped and drifted dimer models are the same.   Therefore, for any $A\in \mathcal{F}$, $\mu_f (A)=\mu_d(A)$. This means that on the plane $\mu_f = \mu_d$ almost surely.

\end{proof}

An immediate corollary of Theorem \ref{intro:thm:equivmeas} is

\begin{coro}
The height function defined from $\mu_f$ has the same distribution as the height function defined from $\mu_d$.
\end{coro}

\subsection{Bijections with other Models} \label{sec:discrete:bij}


Denote $\mathtt{ST}(n,e,s,w)$ to be the (random) directed spanning tree with weights $n=n(y_1,y_2)$ upwards, $e=e(y_1,y_2)$ eastwards, etc.  These weights can be dependent on the position of the edge.  The language of spanning trees is similar to the language of the dimer model and so to avoid ambiguity, we refer to the vertices and directed edges of the directed spanning trees as nodes and arrows.

The drifted dimer model, $\mu_d(s_1,s_2,s_3,s_4)$ is equal in distribution with $\mathtt{ST}(s_1,s_2,s_3,s_4)$.   The flipped dimer model, $\mu_f(r_1,r_2,r_3,r_4)$, is equal in distribution to another spanning tree model, which is gauge equivalent to $\mathtt{ST}(s_1^{y_1+y_2},s_2^{y_1+y_2},s_3^{y_1+y_2},s_4^{y_1+y_2})$ which is P-equivalent to $\mathtt{ST}(s_1,s_2,s_3,s_4)$.

Let $\mu_{SO}(t)$ denote the measure of the dimer model on the square octagon lattice (with a specified boundary conditions) with edge weights $t$ on the diagonal edges and edge weight 1 on the remaining edges.  The case $t=\sqrt{2}$ corresponds to the dimer model on the square octagon lattice being at criticality.    Let $t=s^2/2$.

\begin{lemma} \label{thermo:lem:squareoct}
$\mu_{SO}(s^2/2)$  with any boundary conditions is equal in distribution to $\mu_f(1,s,s,1)$ and P-equivalent to $\mu_d(1,s^2,s^2,1)$ and $\mathtt{ST}(1,1,s^2,s^2)$, with the appropriate boundary conditions.
\end{lemma}

\begin{proof}
Most of the proof follows from the corresponding result for the uniform measure in \cite{kpw:00}.  However, we do provide details in obtaining the appropriate weights of the spanning tree.   The first step is to decompose the square octagon lattice to a square grid using urban renewal.  For the square octagon lattice labelled in Figure \ref{setup:fig:fund}, allocate the vertices $b_1,w_1,b_2,w_2$ to be {\it rich cities} and squares with the vertices $b_3,w_3,b_4,w_4$ to be {\it poor cities}.  We can decompose each poor city along with the connecting edge and vertex of each of the four neighboring rich cities into a {\it new city} with a different weight, say $s$, along each edge in the new city as shown in Figure \ref{trees:urban}. 

	\begin{figure}[h!]
		\begin{center}
		\scalebox{1}{\includegraphics{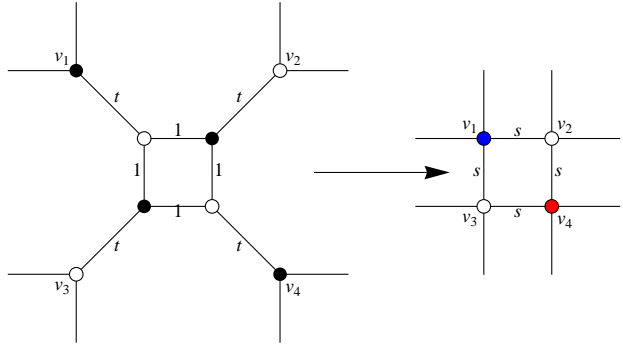}}
		\caption[Urban Renewal for the Square Octagon lattice]{The urban renewal transform, which transforms a poor city (neighbored by 4 rich cities), into a new city, with weights $s$ on the edges}
		\label{trees:urban}
		\end{center}
	\end{figure}

For the edges surrounding a poor city, there are either $0, 2$ or $4$ of the edges covered by dimers.  In the decomposition by urban renewal, these correspond to no edges, 1 edge or two edges covered by dimers in the new city.  As we have set $t=s^2/2$, it is clear that the weights of the local configurations agree. 

We color in blue the upper left hand vertices of each new city.  We color in red the lower right hand vertices of each new city.  The red vertices form the set of nodes for the spanning tree while the blue vertices form the nodes of the associated dual spanning tree.   In order to obtain the appropriate weights on the arrows of the spanning tree, we must first apply a gauge transformation so that all the edges incident to a blue vertex have weight 1.  An instance of this gauge transformation is shown in Figure \ref{trees:urban2}.  The gauge transformed dimer model has different weights depending on the locations of the edges, that is, the dimer model is $\mu_d(s^{-(y_1+y_2)},s^{-(y_1+y_2+2)},s^{-(y_1+y_2+2)},s^{-(y_1+y_2)})$.  It is clear that this dimer model is P-equivalent $\mu_d(1,s^2,s^2,1)$.
	\begin{figure}[h!]
		\begin{center}
		\scalebox{1}{\includegraphics{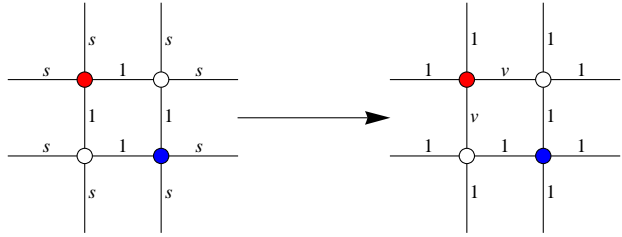}}
		\caption[Gauge transform for the dimer model on the square grid]{The gauge transformation, which applies a transformation on the edge weights.  Here, $v=s^2$.}
		\label{trees:urban2}
		\end{center}
	\end{figure}

\end{proof}

The significance of the colored vertices in the above proof is that the blue vertices are the vertex set of the directed spanning tree and the red vertices are the vertex set of the dual spanning tree.  A dimer covering a blue vertex and a white vertex corresponds to a directed edge from that blue vertex passing through the white vertex to the next blue vertex.

The urban renewal transformation guarantees that the height function defined on the dual of the rich cites is the same as the height function on the dual of the square grid.  More precisely,

	\begin{lemma} \label{trees:urbanunchange}
The urban renewal transformation is a measure preserving bijection for the height function.
	\end{lemma}

\begin{proof}
The height function is defined up to height difference which is determined by the presence of a dimer in the configuration.  As the urban renewal transformation is measure preserving, then it remains to check that the height change across the edges are the same.  This can be validated by checking each local configuration.
\end{proof}


\subsection{Massive and Drifted Operators}

In this subsection we introduce the {\it massive} and {\it drifted} discrete operators.  Applying these operators to the inverse Kasteleyn entries gives a discrete partial differential equation with the boundary conditions given by the graph of the underlying dimer model.  For $v \in \mathtt{W}$, define the Kasteleyn operator by 
\(
	K f(v) = \sum_{u \sim v} K(v,u) f(u)
\)
where $u \sim v$ means that $u$ and $v$ are nearest neighbors, and  $(v,u)^{th}$ entry of the $K(u,v)$ is the Kasteleyn matrix.  For $v \in \mathtt{B}$, define the dual Kasteleyn operator by 
\(
	K^* f(v)=\sum_{u \sim v} K(v,u) f(u)
\)
Let $e_1=(2,0)$ and $e_2=(0,2)$. The {\it massive operator} is given by
\begin{align}
	L_f f&=K^* K f(v) \nonumber \\
	&=r_2r_4 f(v+e_1)+r_2 r_4 f(v-e_2)+r_1r_3 f(v+e_2)+r_1r_3 f(v-e_2) - (r_1^2+r_2^2+r_3^2+r_4^2).  \label{thermo:killedoperator}
\end{align}
Define the shifted graph of $\Z^2$ to be the graph on $\Z^2$ with the weights of the horizontal edges shifted upwards by one lattice spacing and the weights of the vertical edges shifted to the right by one lattice spacing, as seen in Figure \ref{killedanddrift:kkstar}.
	\begin{figure}[h!]
		\begin{center}
		\scalebox{1}{\includegraphics{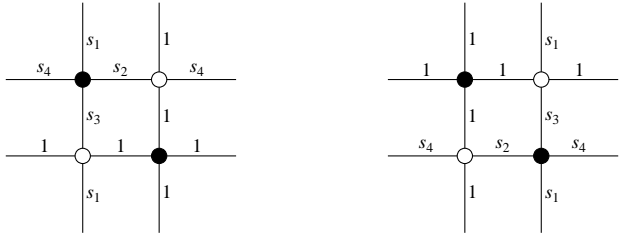}}
		\caption{The original graph and the shifted graph for the flipped dimer model.}
		\label{killedanddrift:kkstar}
		\end{center}
	\end{figure}
For $v \in \mathtt{B}$, define the dual Kasteleyn operator on the shifted graph by
\(
 	\tilde{K} f(v)= \sum_{u \sim v} K(v,u) f(u) 
\)
where $K(v,u)$ is the $(v,u)^{th}$ entry of the Kasteleyn matrix on the the shifted graph.  The {\it drifted operator} is given by
\begin{align}
	L_df(v)&=\tilde{K} K f(v) \nonumber \\
	&= s_1f(v+e_2) + s_2 f(v+e_1) + s_3 f(v-e_2)+s_4 f(v-e_1) - (s_1+s_2+s_3+s_4).  \label{thermo:driftedoperator}
\end{align}

\section{Scaling Window} \label{sec:scalingwindow}

This section focuses on the behavior of the height function from the flipped dimer model.  In doing so, we have to introduce the Green's function of the massive random walk as well as the Green's function of the  drifted random walk.    These functions allow the computation of the inverse Kasteleyn entries in the scaling window, which can be used to show that the fluctuations in the scaling window form a non-Gaussian field. 

\subsection{Green's Functions}

In this subsection, we compute the massive and drifted Green's functions in the scaling window.   These functions are closely related to the flipped and drifted dimer model -- the inverse Kasteleyn entries for the flipped (and respectively drifted) dimer model can be written as a linear combination of massive (drifted) Green's functions.

\subsubsection{Drifted Green's Function} 

The drifted operator, defined in (\ref{thermo:driftedoperator}), can be re-written as
\(
	L_d f(v)=s_1\partial_y f(v)+s_2 \partial_x f(v)-s_3 \partial_y f(v-e_2) -s_4 \partial_x f(v-e_1)
\)
where $\partial_x$ and $\partial_y$ represent the discrete derivatives.  Let $H^1_k$ be the Green's function of the $L_d^1$ on the plane, where the superscript represents the lattice spacing.  In other words, we have
\(
	-L_d^1 H^1_d=\delta_0. \label{scale:eqn:driftpde}
\)

There is a finite drift in the scaling window $(\e, \e)$ if $s_1 - s_3$ and $s_2- s_4$ are both $O(\e)$.  Choosing $s_1$ and $s_4$ equal to $1$, we have

\begin{lemma} \label{scale:lem:driftedgreens}
Suppose that $\lim_{\e \to 0} 1/\e (s_3-s_1)=\l_2>0$ and $\lim_{\e \to 0} 1/\e (s_4-s_2)=\l_1>0$.  In the scaling window $(\e, \e)$,  for $(x,y) \in \R^2 \backslash \{(0,0)\}$
\(
	H^0_d(x,y)= \frac{1}{\pi} \exp \left(  \l_1 x-\l_2 y\right)	B_K\left( 0 , \sqrt{(\l_1^2 +\l_2^2)(x^2+y^2)} \right)
\)
where $B_K(n,z)$ is the modified Bessel function of the second kind with index $n$.
\end{lemma}

\begin{proof}
For the above choice of weights, the operator $L^\e_d$ is the Markov generator of a two-dimensional random walk with drift $(\l_1 \e,-\l_2 \e)$  because
\(
	L_d^\e f(v) = \Delta^\e f(v) + \l_2 \e (f(v-e_1)-f(v)) +\l_1 \e (f(v-e_2)-f(v)) 
\)
where $\Delta^\e$ is the discrete Laplacian operator on $\e \Z^2$.  By Donsker's theorem, this random walk with drift converges (weakly) to a Brownian motion with drift $\underline{\l}:=(\l_1,-\l_2)$ and the partial differential equation now reads
\(
	(-\Delta+\underline{\l} \cdot \nabla ) H_d^0=\delta_0. 
\)
As a two-dimensional Brownian motion with drift is transient, the Greens' function  can be computed using the usual probabilistic methods, namely, for $y\in \R^2 \backslash \{0\}$
\begin{align}
	H_d^0(0,y) &= \int_0^\infty p(0,y,t) dt\\
	 		&= \int_0^\infty \frac{1}{2\pi t} \exp\left( -\frac{|y-\underline{\l} t |^2}{2t} \right) dt \label{thermo:eqn:driftgreens} 
\end{align}
where $p(0,y,t)$ is the transition kernel for the Brownian motion with drift.  The integral in (\ref{thermo:eqn:driftgreens}) is a standard integral and is exactly the Bessel Function given in the lemma.

\end{proof}

The function $H_d^0$ is the Green's function for a two dimensional Brownian motion with drift on the plane.  This is the reason behind the name - `the drifted dimer model'.

\subsubsection{Green's function of the Massive Random Walk} \label{sec:subsubkilledgreens}

The operator defined in (\ref{thermo:killedoperator}) is the (Markov) generator of a massive random walk.  The random walk has transition probabilities $(r_1 r_3, r_2 r_4, r_1 r_3, r_2 r_4)/(r_1^2 +r_2^2 +r_3^2 +r_4^2)$ for the north, east, south and west steps and probability $((r_1-r_3)^2+(r_2-r_4)^2)/(r_1^2 +r_2^2 +r_3^2 +r_4^2)$ of visiting the so-called {\it graveyard}.

The massive operator (\ref{thermo:killedoperator}) can be rewritten as 
\(
	L_f f = \left( r_2 r_4\partial_y^2 + r_1 r_3 \partial_x^2 -(r_1-r_3)^2-(r_2-r_4)^2 \right) f
\)
where $\partial_x$ and $\partial_y$ represent the discrete derivatives.  Let $H^1_f$ be the Green's function of $L_f^1$ on the plane, where the superscripts represents the lattice spacing.  In other words, 
\(
	-L_f^1 H^1_f = \delta_0.\label{scalep:eqn:dispde1}
\)
Note that we can also write 
\(
	H^1_f(x,y)=\frac{1}{(2\pi i)^2}  \int_{|w|=1} \int_{|z|=1} \frac{ w^x z^y}{P_f(z,w)} \frac{dz}{z}\frac{dw}{w}  \label{scalep:eqn:dispde2}
\)
i.e. $H^1_f$ is the Fourier Coefficients of $1/P_f(z,w)$.  This can be seen by evaluating the appropriate linear combination of $L_f$.  

In order for the massive random walk to have a  finite rate of visiting the graveyard  in the scaling window $(\e,\e)$, we require $r_1 - r_3$ and $r_2 -r_4$ to be $c \e$ where $c$ is some constant.  Choosing $r_1$ and $r_4$ equal to 1, we have

\begin{lemma} \label{scale:lem:killedgreens}
Suppose that $\lim_{\e \to 0} 1/\e (r_3-r_1)=\l_1>0$ and $\lim_{\e \to 0} 1/\e (r_4-r_2)=\l_2>0$.  In the scaling window $(\e, \e)$,  for $(x,y) \in \R^2 \backslash \{(0,0)\}$
\(
	H^0_f(x,y)= \frac{1}{\pi} B_K\left( 0 , \sqrt{(\l_1^2 +\l_2^2)(x^2+y^2)} \right)
\)
where $B_K(n,z)$ is the modified Bessel function of the second kind with index $n$.
\end{lemma}

\begin{proof}
Set $G_1(m,n)=H^1_d(m,n) r_2^m r_3^{-n}$ for $(m,n)\in \Z^2$.  $G_1$ saftisfies the discrete partial differential equation (\ref{scalep:eqn:dispde1}) and so $G_1 = H^1_f$.  $H_d^0$ is computed in Lemma \ref{scale:lem:driftedgreens} and evaluating the exponential factors gives $H^0_f$.

\end{proof}

If $r_1$ and $r_4$ do not approach $1$ in the scaling window but $r_1 -r_3$ and $r_2-r_3$ are $O(\e)$, then the Green's function in the scaling window does exist.  It is given by a stretch or contraction in the appropriate direction of the Green's function given in Lemma \ref{scale:lem:killedgreens}.

\subsection{Non-Gaussian Scaling Window} \label{subsec:nongauss}
In this subsection, we show that the height function for the flipped dimer model with specific edge weights converges to a non-Gaussian field, i.e. we prove Theorem \ref{intro:thm:non-gauss}.   In order to prove this assertion, we need to determine the correlation functions for the flipped dimer model in the scaling window.  Theorem \ref{intro:thm:non-gauss} follows by showing that the fourth moment of the height function cannot be written as the sum of a product of the second moments of the height function and hence violates Wick's theorem.  

The next lemma gives the explicit expressions for the correlation function of the dimer model.

\begin{lemma}\label{scale:lem:invkastentries}
Let $|\underline{\l}| = \sqrt{\l_1^2 + \l_2^2}$, $z=(x,y)$ and $|z|=\sqrt{x^2+y^2}$, with $\lim_{\e \to 0} 1/\e(b_i-w_j)=(x,y)$, where $b_i \in \mathtt{B}_i$ and $w_j \in \mathtt{W}_j$ for $i,j \in \{0,1\}$.  In the scaling window $(\e,\e)$, the entries of the inverse Kasteleyn matrix are given by
\(
	\lim_{\e \to 0} \frac{1}{\e} K^{-1} (b_0, w_0)= \frac{|\underline{\l}|}{\pi|z|} x B_K\left(1, |\underline{\l}| |z| \right) + \frac{1}{\pi} \l_1 B_K \left( 0,  |\underline{\l}| |z| \right),
\)
\(
	\lim_{\e \to 0} \frac{1}{\e} K^{-1} (b_0, w_1)=- \frac{|\underline{\l}|}{\pi |z|} y B_K\left(1,  |\underline{\l}| |z| \right) +\frac{1}{\pi} \l_2 B_K \left( 0,  |\underline{\l}| |z| \right)
\)
\(
	\lim_{\e \to 0} \frac{1}{\e} K^{-1} (b_1, w_0)= \frac{|\underline{\l}|}{\pi |z|} y B_K\left(1,  |\underline{\l}| |z| \right) + \frac{1}{\pi}\l_2 B_K \left( 0,  |\underline{\l}| |z| \right)
\)
and
\(
	\lim_{\e \to 0} \frac{1}{\e} K^{-1} (b_1, w_1)= \frac{|\underline{\l}|}{\pi  |z|} x B_K\left(1,  |\underline{\l}| |z| \right) -\frac{1}{\pi}\l_1 B_K \left( 0,  |\underline{\l}| |z| \right).
\)
where $B_K(n,r)$ is the modified Bessel functions of the second kind with index $n$. 
\end{lemma}

\begin{proof} 
The entries of the inverse Kasteleyn matrix are given by the local statistics formula (\ref{setup:localstats}) where 
\(
	\tilde{K} (z,w)= \left( \begin{array}{cc} 
	r_2 -w & r_3 - z \\ 
	r_3-z^{-1} & -r_2 +w^{-1}  \end{array} \right).
\)
Each entry of the inverse Kasteleyn matrix can be written as a linear combination of the massive Green's function (in the discrete setting).  For example, with lattice spacing $\e$, $K^{-1} (b_0, w_0)$ is given by
\(
	-r_2 H^\e_k(x,y) + H^\e_d(x-\e,y).
\)	
Using a Taylor expansion and Lemma \ref{scale:lem:killedgreens}, this can be rewritten in terms of Bessel functions.   For example,  $K^{-1} (b_0, w_0)$ now reads
\(
	\e \left( -\frac{\partial}{\partial x} H^0_k(x,y) +\l_1 H^0_k(x,y) \right) +O(\e^2).
\)
Dividing through by $\e$ and taking limits gives the results.

\end{proof}

Let $S_n (v_1,\dots,v_n)$ denote the $n$ point correlation function for the height change along the edges $v_1,\dots,v_n$ in the scaling window $(\e, \e)$.  Suppose that $\g_1,\dots, \g_n$ are $n$ disjoint lines in the scaling window $(\e,\e)$, we let
 \(
 	S_n(\g_1,\dots,\g_n)= \int_{\g_1} \dots \int_{\g_n} S_n(v_1,\dots, v_n) dv_1 \dots dv_n.
\)

Let $\gamma_1, \dots , \gamma_4$ be four lines (in the scaling window) on the same vertical line with $\g_i \cap \g_j = \emptyset$.  Any two point correlation function can be determined for the lines $\g_1, \dots , \g_4$

\begin{lemma} \label{scale:lem:twopoint}
In the scaling window $(\e,\e)$, the two point correlation function for the height function along  $\g_i$ and $\g_j$ for $i \not= j$ and $i,j \in \{1,2,3,4\}$ is given by
\(
	S_2 (\g_i,\g_j)=\frac{18}{\pi^2} |\underline{\l}|^2\int_{\g_i} \int_{\g_j}  B_K(0,|\underline{\l}| |v_i-v_j|)^2+  B_K(1,|\underline{\l}||v_1-v_2|)^2 dv_i dv_j. 
\)
\end{lemma}

\begin{proof}
The definition of the height function gives

\(	\label{scale:eqn:twopointdef}
	S_2(\g_i, \g_j)=\lim_{\e\to0} 3^2\E \left[ \sum_{k=0}^{2n} \sum_{l=0}^{2m} (X_{2k}-X_{2k+1}) (X_{2l} -X_{2l+1}) \right]
\)
where the index $k$ and $l$ are the discretized index of the lines $\g_i$ and $\g_j$, and $X_k$ is the indicator event that there is a horizontal dimer crossing $\g_i$ at index $k$.   Interchanging the sums,  we are required to compute 
\(
	\E \left[(X_{2i}-X_{2i+1}) (X_{2j} -X_{2j+1}) \right].
\)
This quantity is a linear combination of four 2 by 2 determinants, i.e.
\(
	\sum_{i_1,j_1=0}^1 (-1)^{i_1+j_1}\det \left( \begin{array}{cc}
	0 & K^{-1}(w_{i_1}, b_{j_1}) \\
	K^{-1}(w_{j_1}, b_{i_1}) & 0 \end{array} \right) 
\)
where $\lim_{\e \to 0} 1/\e|w_{m}- b_{n}|=|m-n|$.  These entries can approximate by Lemma \ref{scale:lem:invkastentries}.  This gives
\(
	-2\e^2 \frac{ |\underline{\l}|^2}{\pi^2} \left( B_K(0,|\underline{\l}|^2|v_j-v_i|)^2+ B_K(1,|\underline{\l}|^2|v_j-v_i|)^2 \right) \label{scale:lempro:integrand}
\)
where $v_1=k+O(\e)$ and $v_2=l+O(\e)$.  Plugging (\ref{scale:lempro:integrand}) back into (\ref{scale:eqn:twopointdef}) and taking the limit as $\e \to 0$, the Riemann sum converges to the required integral.

\end{proof}

In a similar fashion, we can find an exact expression for the four point correlation function for the change of heights along $\g_1, \dots, \g_4$.  However, this expression is rather complicated, and the full expression is not required for the proof of  Theorem \ref{intro:thm:non-gauss}.  Instead, we just require the following

\begin{lemma} \label{scale:lem:fourpoint}
In the scaling window $(\e,\e)$, the four point correlation function for the height function along $\g_1,\dots,\g_4$ is given by
\(
	S_4(\g_1, \dots, \g_4)-\sum_{i,j=1}^4 \prod_{i \not = j} S_2(\g_i,\g_j)  = \int_{\g_1} \dots \int_{\g_4} f_{(\l_1,\l_2)}(x_1,x_2,x_3,x_4) dx_4 \dots dx_1. \label{scale:fourpoint}
\)
where $f_{(\l_1,\l_2)}(x_1,x_2,x_3,x_4)$ is defined in the Appendix equation (\ref{appendix:mess}).  When $(\l_1,\l_2)>0$, we have $f_{(\l_1,\l_2)}(x_1,x_2,x_3,x_4) \not =0$.  When $(\l_1,\l_2) \to 0$, we have $f_{(\l_1,\l_2)}(x_1,x_2,x_3,x_4) \to 0$.
\end{lemma}

\begin{proof}
	From the definition of the height function, we can evaluate $S_4(\g_1, \g_2,\g_3,\g_4)$, using the four point correlation function of the dimer model, namely,
\(	
	\lim_{\e\to0} 3^4\E \left[ \sum_{i=0}^{2n} \sum_{j=0}^{2m} \sum_{k=0}^{2p} \sum_{l=0}^{2q} (X_{2i}-X_{2i+1}) (X_{2j} -X_{2j+1}) (X_{2k}-X_{2k+1})(X_{2l}-X_{2l+1}) \right] \label{scale:lempro:fourpoint}
\)
where the index $i,j,k,l$ are the discretized index of the lines $\g_1,\dots ,\g_4$, and $X_r$ is the indicator event that there is a horizontal dimer crossing the corresponding vertical line. We can approximate 
\(
	\E \left[(X_{2i}-X_{2i+1}) (X_{2j} -X_{2j+1}) (X_{2k}-X_{2k+1})(X_{2l}-X_{2l+1}) \right]
\)
using Lemma \ref{scale:lem:invkastentries} and then follow the same approach given in Lemma \ref{scale:lem:twopoint}, i.e. we can write
\begin{align}
	&\sum_{i_1,j_1,k_1,l_1=0}^1 (-1)^{i_1+j_1+k_1+l_1} \nonumber \\
	&\det \left( \begin{array}{cccc}
	0 & K^{-1}(w_{i_1}, b_{j_1})  & K^{-1}(w_{i_1},b_{k_1}) & K^{-1} ( w_{i_1},b_{l_1})   \\
	K^{-1}(w_{j_1}, b_{i_1}) & 0 & K^{-1}(w_{j_1},b_{k_1}) & K^{-1} ( w_{j_1},b_{l_1}) \\
	K^{-1}(w_{k_1}, b_{i_1}) & K^{-1}(w_{k_1},b_{j_1}) &0 & K^{-1} ( w_{k_1},b_{l_1}) \\
	K^{-1}(w_{l_1}, b_{i_1}) & K^{-1}(w_{l_1},b_{j_1}) & K^{-1} ( w_{l_1},b_{k_1}) & 0 
	\end{array} \right) 
\end{align}
with $\lim_{\e \to 0} 1/\e |w_{m_1}-b_{n_1}|=|m-n|$.  Each of the entries of the above matrix can be approximated using Lemma \ref{scale:lem:invkastentries} and each term contains a factor or $\e$.  We can expand out the above determinants and (\ref{scale:lempro:fourpoint}) is a Riemann sum which converges in the limit $\e \to 0$ to a sum of products of Bessel functions.

In order to prove Lemma \ref{scale:lem:fourpoint}, it remains to subtract the sum of the products of the two point correlation function.  This leaves the integral of some linear combination of Bessel functions which are exactly the contributions of the 4-cycles in the above 4 by 4 determinant, i.e. $ \int_{\g_1} \dots \int_{\g_4} f_{(\l_1,\l_2)}(x_1,x_2,x_3,x_4) dx_4 \dots dx_1$.  The full expression of $f_{(\l_1,\l_2)}(x_1,x_2,x_3,x_4)$ can be found in the Appendix, equation (\ref{appendix:mess}).  Each term of $f_{(\l_1,\l_2)}(x_1,x_2,x_3,x_4)$ is of the form 
\(
	\pm 4\frac{ |\l|^4}{\pi^4} B_K(i_1, |\l| z_1)B_K(i_2, |\l| z_2)B_K(i_3, |\l| z_3)B_K(i_4, |\l| z_4)
\)
where each $i_k \in \{ 0, 1\}$ for $1 \leq k \leq 4$, $i_1 +i_2 +i_3 +i_4$ is divisible by 2 and $z_k=x_r-x_s$ for $1\leq s<r\leq 4$ .  Notice that $f_{(\l_1,\l_2)}(x_1,x_2,x_3,x_4)$ is a function of $\sqrt{\l_1^2+\l_2^2}$ and is non-zero for $|\lambda|>0$. 

Taking the generalized series expansion around $|\l|=0$ (using the generalized series expansion of the Bessel functions), as  $B_K(1,z)$ is approximately $1/z+z/4 (-1+2E_G-\log4+2\log(z)) +O(z^2)$ and $B_K(0,z)=\log(z)-E_G +\log(2)+O(z^2)$ where $E_G$ is Euler's constant, we find that the highest order (with respect to $|\l |$) of the integrand is  $O(-| \l |^2 \log | \l |)$.  Using the dominate convergence theorem, the left hand side of (\ref{scale:fourpoint})   tends to zero as $|\l|$ tends to zero.

\end{proof}

\begin{proof}[Proof of Theorem \ref{intro:thm:non-gauss}]
Lemma \ref{scale:lem:fourpoint} and Lemma \ref{appendixlemma} given in the Appendix shows that we can find intervals $\g_1, \dots,\g_4$ such that
\(
	S_4(\g_1, \dots, \g_4) \not = \sum_{i,j=1}^4 \prod_{i \not = j} S_2(\g_i,\g_j.)
\)
This means that the correlation functions do not satisfy Wick's theorem (that the moments of the normal can be written as the permanent of the Green's function).  This proves the theorem.
\end{proof}

\begin{coro}
For $\l>0$, in the scaling window $(\e, \e)$ the height fluctuations of $\mu_{SO}((1-\l \e)^2)$ on the plane are non-Gaussian. 
\end{coro}

\begin{proof}
By Lemma \ref{thermo:lem:squareoct} and Lemma \ref{trees:urbanunchange}, the height fluctuations of $\mu_{SO}((1-\l \e)^2)$ are equal in distribution to the height fluctuations of $\mu_f (1,1-\l \e,1-\l \e,1)$, which are non-Gaussian by Theorem \ref{intro:thm:non-gauss}.

\end{proof}

The first part of Theorem \ref{intro:thm:non-gauss} also holds when $r_1 \not \to 1$, $r_4 \not \to 1$ as $\e \to 0$ but $r_1-r_3$ and $r_2-r_4$ is $O(\e)$.  If $r_1 \to k_2$ and $r_4 \to k_1$, the argument in the Green's function becomes 
\(
	\half \sqrt{\l_1^2+\l_2^2}\sqrt{x^2 k_1^2 + y^2 k_2^2}.
\)

 The arguments in the Green's function change by a multiplicative factor and hence, the derivative terms in the inverse Kasteleyn entries change by multiplicative factors.  However, the same proof for the first part of Theorem \ref{intro:thm:non-gauss} holds, but with these extra multiplicative factors attached to the derivatives of the Green's function.

\subsection{Amoeba Interpretation}

The amoeba for the flipped dimer model $\mu_f(1,1-\l_1 \e, 1-\l_2 \e,1)$, consists of an ellipse with the intercept with the line $e^{|w|}=0$ given by 

\(
	\pm \log \left[\frac{2-2 \e \lambda _2+\e^2 \left(\lambda _1^2+\lambda _2^2\right)-\sqrt{\e^2 \left(\lambda _1^2+\lambda _2^2\right) \left(\e^2 \lambda _1^2+\left(2-\e \lambda _2\right){}^2\right)}}{2-
2 \e \lambda _2}\right]
\)
and the intercept with the line $e^{|z|}=0$ given by 
\(	
 	\pm \log \left[\frac{2+2 \epsilon  \lambda _1+\epsilon ^2 (\lambda _1^2+ \lambda _2^2)-\sqrt{\e^2 \left(\lambda _1^2+\lambda _2^2\right) \left(\e^2 \lambda _2^2+\left(2+\e \lambda _1\right){}^2\right)}}{2+2 \epsilon  \lambda _1}\right].
\)
An expansion of these terms gives 
\(
	\pm \e \sqrt{\l_1^2+\l_2^2}\pm \e^2 \l_2 \sqrt{\l_1^2 + \l_2^2}
\)
and 
\(
	\pm \e \sqrt{\l_2^2+\l_2^2}\pm \e^2 \l_1 \sqrt{\l_1^2 + \l_2^2}.
\)
Rescaling the ellipse (by $\e$) and as $\e$ tends to zero, the ellipse tends to a circle with radius $\sqrt{\l_1^2+\l_2^2}$.    For measures of the form $\mu_f(r_1(\e),r_2(\e),r_3(\e),r_4(\e))$, with $\lim_{\e \to 0} r_1 \not = \lim_{\e \to 0} r_4$, the bounded region is an ellipse, with foci $O(\e)$ provided $r_1 - r_3$ and $r_2 -r_4$ is $O(\e)$.  

As the height fluctuations of $\mu_f(1,1-\l_1 \e, 1-\l_2 \e,1)$ and $\mu_d(1,(1-\l_1 \e)^2, (1-\l_2 \e)^2,1)$ are equivalent, then the fluctuations at the boundary of the gaseous region and the center of the gaseous region are equivalent in the scaling window.  In fact, the height fluctuations of the whole gaseous region is equivalent which can be seen by gauge transformations.

\subsection*{Acknowledgements}:  I would like to particularly thank Richard Kenyon for the many fruitful discussions which have led to this paper. I would also like to thank David Brydges for discussions of statistical mechanical models, Scott Sheffield for some very useful suggestions, C\'{e}dric Boutillier, Benjamin Young and Adrien Kassel for very many useful comments on this paper. Supported/Partially supported by the grant KAW 2010.0063 from the Knut and Alice Wallenberg Foundation.

\begin{appendix}

\section{}

In the appendix, we present the function from Lemma \ref{scale:lem:fourpoint}.  For $y_1=x_2-x_1$, $y_2=x_3-x_2$ and $y_3=x_4-x_3$, we have

\begin{align}
& \pi^4 f_{\l_1,\l_2}(x_1,x_2,x_3,x_4) = \nonumber \\
&-4 |\l|^4 B_K[0, |\l| y_1] B_K[0, |\l| (y_1 + y_2)] B_K[0,  |\l| y_3] B_K[0, |\l| (y_2 + y_3)] \nonumber \\ &- 
  4 |\l|^4 B_K[0, |\l| y_1] B_K[0, |\l| y_2] B_K[0, |\l| y_3] B_K[0,  |\l| (y_1 + y_2 + y_3)] \nonumber \\&- 
  4 |\l|^4 B_K[0, |\l| y_2] B_K[0, |\l| (y_1 + y_2)] B_K[0, |\l| (y_2 + y_3)] B_K[0, |\l| (y_1 + y_2 + y_3)] \nonumber \\ &+ 
  4 |\l|^4 B_K[0, |\l| y_3] B_K[0, |\l| (y_1 + y_2 + y_3)] B_K[1, |\l| y_1] B_K[1, |\l| y_2] \nonumber \\ &- 
  4 |\l|^4 B_K[0, |\l| y_3] B_K[0, |\l| (y_2 + y_3)] B_K[1, |\l| y_1] B_K[1, |\l| (y_1 + y_2)]\nonumber \\ & - 
 4 |\l|^4 B_K[0, |\l| (y_2 + y_3)] B_K[0, |\l| (y_1 + y_2 + y_3)] B_K[1, |\l| y_2] B_K[1, |\l| (y_1 + y_2)] \nonumber \\ &+ 
  4 |\l|^4 B_K[0, |\l| (y_1 + y_2)] B_K[0, |\l| (y_2 + y_3)] B_K[1, |\l| y_1] B_K[1, |\l| y_3] \nonumber \\ & - 
  4 |\l|^4 B_K[0, |\l| y_2] B_K[0, |\l| (y_1 + y_2 + y_3)] B_K[1, |\l| y_1] B_K[1, |\l| y_3]  \nonumber \\ & + 
  4 |\l|^4 B_K[0, |\l| y_1] B_K[0, |\l| (y_1 + y_2 + y_3)] B_K[1, |\l| y_2] B_K[1, |\l| y_3] \nonumber \\ &+ 
  4 |\l|^4 B_K[0, |\l| y_1] B_K[0, |\l| (y_2 + y_3)] B_K[1, |\l| (y_1 + y_2)] B_K[1, |\l| y_3] \nonumber \\ &+ 
  4 |\l|^4 B_K[0, |\l| (y_1 + y_2)] B_K[0, |\l| y_3] B_K[1,  |\l| y_1] B_K[1, |\l| (y_2 + y_3)] \nonumber \\ &- 
  4 |\l|^4 B_K[0, |\l| (y_1 + y_2)] B_K[0, |\l| (y_1 + y_2 + y_3)] B_K[1, |\l| y_2] B_K[1, |\l| (y_2 + y_3)] \nonumber \\ &+ 
  4 |\l|^4 B_K[0, |\l| y_1] B_K[0, |\l| y_3] B_K[1,  |\l| (y_1 + y_2)] B_K[1, |\l| (y_2 + y_3)]\nonumber \\ & - 
  4 |\l|^4 B_K[0, |\l| y_2] B_K[0, |\l| (y_1 + y_2 + y_3)] B_K[1,  |\l| (y_1 + y_2)] B_K[1, |\l| (y_2 + y_3)] \nonumber \\ &- 
  4 |\l|^4 B_K[0, |\l| y_1] B_K[0, |\l| (y_1 + y_2)] B_K[1, |\l| y_3] B_K[1, |\l| (y_2 + y_3)] \nonumber \\ &- 
  4 |\l|^4 B_K[1, |\l| y_1] B_K[1, |\l| (y_1 + y_2)] B_K[1,  |\l| y_3] B_K[1, |\l| (y_2 + y_3)] \nonumber \\ &- 
  4 |\l|^4 B_K[0, |\l| y_2] B_K[0, |\l| y_3] B_K[1, |\l| y_1] B_K[1,  |\l| (y_1 + y_2 + y_3)] \nonumber \\ &+ 
  4 |\l|^4 B_K[0, |\l| y_1] B_K[0, |\l| y_3] B_K[1, |\l| y_2] B_K[1, |\l| (y_1 + y_2 + y_3)]\nonumber \\ & - 
  4 |\l|^4 B_K[0, |\l| (y_1 + y_2)] B_K[0, |\l| (y_2 + y_3)] B_K[1,  |\l| y_2] B_K[1, |\l| (y_1 + y_2 + y_3)] \nonumber \\ &- 
  4 |\l|^4 B_K[0, |\l| y_2] B_K[0, |\l| (y_2 + y_3)] B_K[1, |\l| (y_1 + y_2)] B_K[1, |\l| (y_1 + y_2 + y_3)] \nonumber \\ &- 
  4 |\l|^4 B_K[0, |\l| y_1] B_K[0, |\l| y_2] B_K[1, |\l| y_3] B_K[1,  |\l| (y_1 + y_2 + y_3)]\nonumber \\ & + 
  4 |\l|^4 B_K[1, |\l| y_1] B_K[1, |\l| y_2] B_K[1, |\l| y_3] B_K[1, |\l| (y_1 + y_2 + y_3)] \nonumber \\ &- 
  4 |\l|^4 B_K[0, |\l| y_2] B_K[0, |\l| (y_1 + y_2)] B_K[1,  |\l| (y_2 + y_3)] B_K[1, |\l| (y_1 + y_2 + y_3)] \nonumber \\ &- 
  4 |\l|^4 B_K[1, |\l| y_2] B_K[1, |\l| (y_1 + y_2)] B_K[1,   |\l| (y_2 + y_3)] B_K[1, |\l| (y_1 + y_2 + y_3)] \label{appendix:mess}
\end{align}

We also have the following lemma which is based on a crude argument. 

\begin{lemma}\label{appendixlemma}
There exists $\g_1, \dots,\g_4$, so that 
\(
 	\int_{\g_1} \dots \int_{\g_4}f_{\l_1,\l_2}(x_1,x_2,x_3,x_4)  dx_1 \dots dx_4 >0
\)

\end{lemma}

\begin{proof} 
To show that $\int_{\g_1} \dots \int_{\g_4}f_{\l_1,\l_2}(x_1,x_2,x_3,x_4)  dx_1 \dots dx_4\not=0$ for some choice of intervals $\g_i$, choose $\g_1=(0,1/2^n)$, $\g_2=(1-1/2^n,1)$, $\g_3=(2,2+1/2^n)$ and $\g_4=(3,3-1/2^n)$.  We claim that $ f_{\l_1,\l_2}(x_1,x_2,x_3,x_4)$ is positive over this interval and so the integral over $\g_1, \dots, \g_4$ is non-zero.  


We can bound $y_1$ and $y_3$ in terms of $y_2$ where $y_1, y_2$ and $y_3$ are defined above.  These are given by 
\(
	 y_2\frac{2^n-2}{2^n+2} \leq y_1\leq y_2
\)
and 
\(
	 y_2\frac{2^n-2}{2^n+2} \leq y_3\leq y_2.
\)

 We can separate (\ref{appendix:mess}) into positive terms and negative terms, i.e. we can write (\ref{appendix:mess}) as $g_+(y_1,y_2,y_3)-g_-(y_1,y_2,y_3)$ where $g_-$ and $g_+$ are positive.  A lower bound for the positive terms can be achieved by setting $y_1$ and $y_3$ to $y_2$.  An upper bound of the absolute value of the negative terms can be achieved by setting $y_1$ and $y_3$ to $y_2(2^{n}-2)/(2^{ n}+2)$. In other words, we have
 \(
 	g_+(y_1,y_2,y_3)>g_+(y_2,y_2,y_2)
 \)
 and
 \(
 	g_-(y_1,y_2,y_3)<g_-\left(\frac{2^{n}-2}{2^{ n}+2}y_2,y_2,\frac{2^{n}-2}{2^{ n}+2}y_2\right)
 \)
 The difference between these two bounds is positive provided we take $n=|\lambda| m$ for $|\lambda| \geq 1$ and $n=m/ |\lambda|$ for $|\lambda|<1$ where $m>20$ (these bounds are not tight).  
 
 Indeed, by setting $(2^n-2)/(2^n+2)=1-x$, we can take a series expansion with respect to $x$ between the positive and negative terms, which is given by
 \(
 	g_+(y_2,y_2,y_2)-g_-(y_2,y_2,y_2) - x g_-'(y_2,y_2,y_2) +O(x^2)
 \)
 where the derivative is with respect to $x$.  The $x^0$ term is strictly positive and is bounded below by $e^{- 12 y_2 |\l |}$. This due to comparing the following positive and negative terms of (\ref{appendix:mess}) and noting that the difference in each case is bounded below by $e^{- 12 y_2 |\l |}$.  These differences are given explicitly by
 \(
 4 |\l |^4 B_K[0, |\l | y_2]^2 B_K[1, 2 |\l | y_2]^2 - 
 4 |\l |^4 B_K[0, |\l | y_2]^2 B_K[0, 2 |\l | y_2]^2 
\)
 \(
 4 |\l |^4 B_K[0, 2 |\l | y_2]^2 B_K[1, |\l | y_2]^2 - 
 4 |\l |^4 B_K[0, |\l | y_2]^3 B_K[0, 3 |\l | y_2]
 \)
 \begin{align}
 & 4 |\l |^4 B_K[0, |\l | y_2] B_K[0, 3 |\l | y_2] B_K[1, |\l | y_2]^2 - 4 |\l |^4 B_K[0, |\l | y_2] B_K[0, 2 |\l | y_2]^2 B_K[0, 3 |\l | y_2]  \nonumber \\ 
&- 8 |\l |^4 B_K[0, 2 |\l | y_2] B_K[0, 3 |\l | y_2] B_K[1, |\l | y_2] B_K[
   1, 2 |\l | y_2]  \nonumber \\ 
   & -4 |\l |^4 B_K[0, |\l | y_2]^2 B_K[1, |\l | y_2] B_K[1, 3 |\l | y_2]  \nonumber \\ 
&- 4 |\l |^4 B_K[0, |\l | y_2] B_K[0, 3 |\l | y_2] B_K[1, 2 |\l | y_2]^2
 \end{align}
 \begin{align}
& 4 |\l |^4 B_K[1, |\l | y_2]^3 B_K[1, 3 |\l | y_2] - 4 |\l |^4 B_K[1, |\l | y_2]^2 B_K[1, 2 |\l | y_2]^2 \nonumber \\ 
&- 4 |\l |^4 B_K[0, 2 |\l | y_2]^2 B_K[1, |\l | y_2] B_K[1, 3 |\l | y_2] \nonumber \\ 
&-8 |\l |^4 B_K[0, |\l | y_2] B_K[0, 2 |\l | y_2] B_K[1, 2 |\l | y_2] B_K[1, 3 |\l | y_2]  \nonumber \\ 
&- 4 |\l |^4 B_K[1, |\l | y_2] B_K[1, 2 |\l | y_2]^2 B_K[1, 3 |\l | y_2].
 \end{align}

For $| \l | \geq 1$, the coefficient of the $x^1$ term in the above series expansion is negative and its absolute value is bounded above by $e^{-|\l | y_2}$ for $y_2 \geq 1$.  As $ x$ is smaller than $1/2^{20 \gamma}$, then $f_{\l_1, \l_2}(x_1,x_2,x_3,x_4)$ is positive over $\g_1, \dots, \g_4$ for $| \l | \geq 1$.  For $0< |\l | <1$, the absolute value of the coefficient of $x^1$ is bounded above by some constant $c>2$ for $y_2 \geq 1$.  As $1/2^{20 \gamma}$ is less than $e^{-12 y_2 \gamma}$, the order $x^1$ term is less than the order $x^0$ term.  This means  $f_{\l_1, \l_2}(x_1,x_2,x_3,x_4)>0$ over $\g_1,\dots, \g_4$.

 \end{proof}

\end{appendix}


\end{document}